\documentclass{amsart}
\usepackage{mathptmx}
\usepackage{pxfonts}
\usepackage{amssymb}
\usepackage{amsfonts}
\usepackage{amsmath}
\usepackage{graphicx}
\usepackage{shadow}
\usepackage{color}
\usepackage[all]{xy}
\usepackage{enumitem}
\usepackage[pagebackref]{hyperref}
\newtheorem{thm}{Theorem}[section]

\newtheorem{lemma}[thm]{Lemma}
\newtheorem{proposition}[thm]{Proposition}

\theoremstyle{definition}
\newtheorem{definition}[thm]{Definition}
\newtheorem{example}[thm]{Example}

\theoremstyle{remark}

\newtheorem{question}[thm]{\bf Question}

\newcommand{\field}[1]{\mathbb{#1}}

\newcommand{\Q }{\field{Q}}
\newcommand{\Z }{\field{Z}}

\DeclareMathOperator{\qf}{qf}

\makeatletter
  \newcounter{xenumi}
  \newenvironment{xenumerate}{%
  \begin{list}{(\alph{xenumi})}{
    \setcounter{xenumi}{1}\usecounter{xenumi}
    \setlength{\parsep}{4\p@ \@plus2\p@ \@minus\p@}
    \setlength{\topsep}{6\p@ \@plus2\p@ \@minus2\p@}
    \setlength{\itemsep}{2\p@ \@plus1\p@ \@minus\p@}
    \setlength{\labelwidth}{0mm}
    \setlength{\labelsep}{2mm}
    \setlength{\itemindent}{2mm}
    \setlength{\leftmargin}{0mm}
    \setlength{\listparindent}{0mm}
  }}{\end{list}}
\makeatother

\begin{document}

\title[On $t$-reductions of ideals in pullbacks]{On $t$-reductions of ideals in pullbacks $^{(\star)}$}

\thanks{$^{(\star)}$ Supported by King Fahd University of Petroleum \& Minerals under Research Grant \# RG1328.}


\author[S. Kabbaj]{S. Kabbaj}
\address{Department of Mathematics and Statistics, King Fahd University of Petroleum \& Minerals, Dhahran 31261, KSA}
\email{kabbaj@kfupm.edu.sa}

\author[A. Kadri]{A. Kadri}
\address{Department of Mathematics and Statistics, King Fahd University of Petroleum \& Minerals, Dhahran 31261, KSA}
\email{g201004080@kfupm.edu.sa}

\author[A. Mimouni]{A. Mimouni}
\address{Department of Mathematics and Statistics, King Fahd University of Petroleum \& Minerals, Dhahran 31261, KSA}
\email{amimouni@kfupm.edu.sa}

\date{\today}



\begin{abstract}
 Let $R$ be an integral domain and $I$ a nonzero ideal of $R$.  An ideal $J\subseteq I$ is a $t$-reduction of $I$ if $(JI^{n})_{t}=(I^{n+1})_{t}$ for some positive integer $n$; and $I$ is $t$-basic if it has no $t$-reduction other than the trivial ones. This paper investigates $t$-reductions of ideals in pullback constructions of type $\square$. Section 2 examines the correlation between the notions of reduction and $t$-reduction in pseudo-valuation domains.  Section 3 solves an open problem on whether the finite $t$-basic and $v$-basic ideal properties are distinct. We prove that these two notions coincide in any arbitrary domain.  Section 4 features the main result, which establishes the transfer of the finite $t$-basic ideal property to pullbacks in line with Fontana-Gabelli's result on P$v$MDs \cite[Theorem 4.1]{FG} and Gabelli-Houston's result on $v$-domains \cite[Theorem 4.15]{GH}. This allows us to enrich the literature with new families of examples, which put  the class of domains subject to the finite $t$-basic ideal property strictly between the two classes of $v$-domains and integrally closed domains.
\end{abstract}
\maketitle

\section{Introduction}

\noindent Throughout, all rings considered are commutative with identity. Let $R$ be a ring and $I$ a proper ideal of $R$. An ideal $J\subseteq I$ is a reduction of $I$ if $JI^{n}=I^{n+1}$ for some positive integer $n$. The notion of reduction was introduced by Northcott and Rees to contribute to the analytic theory of ideals in  Noetherian local rings via minimal reductions. An ideal which has no reduction other than itself is called a basic ideal; and a ring has the finite basic ideal property (resp., basic ideal property) if every finitely generated ideal (resp., every ideal) of $R$ is basic. In \cite{H1,H2}, Hays investigated reductions of ideals in Noetherian rings and Pr\"ufer domains. He provided several conditions for an ideal to be basic. His two main results asserted that a domain $R$ is Pr\"ufer (resp., one-dimensional Pr\"ufer) if and only if $R$ has the finite basic ideal property (resp., basic ideal property).

Let $R$ be a domain and $I$ a nonzero fractional ideal of $R$. The $v$-, $t$-, and $w$-closures of $I$ are defined, respectively, by $I_v:=(I^{-1})^{-1}$, $I_t:=\cup J_v$, where $J$ ranges over the set of finitely generated subideals of $I$, and $I_w=\cap IR_M$ where $M$ ranges over the set of maximal $t$-ideals of $R$. Now, let $\star$ be a star operation on $R$ and $I$ a nonzero ideal of $R$. An ideal $J\subseteq I$ is a $\star$-reduction of $I$ if $(JI^{n})^{\star}=(I^{n+1})^{\star}$ for some positive integer $n$.

In \cite{HKM}, the authors extended Hays' aforementioned results to P$v$MDs; namely, a domain has the finite $w$-basic ideal property (resp., $w$-basic ideal property) if and only if it is a P$v$MD (resp., P$v$MD of $t$-dimension one). They also investigated relations among the classes of domains subject to various $\star$-basic properties. In this vein, the problem of whether the finite $t$- and $v$-basic ideal properties are distinct was left open. In \cite{KK}, the authors investigated the $t$-reductions and $t$-integral closure of ideals establishing satisfactory $t$-analogues of well-known results, in the literature, on the integral closure of ideals and its correlation with reductions. One of their main result \cite[Theorem 3.5]{KK} asserts that the $t$-closure and $t$-integral closure of an ideal coincide in the class of integrally closed domains.

This paper investigates $t$-reductions of ideals in pullback constructions of type $\square$ (defined in Section 4). Section 2 examines the correlation between the notions of reduction and $t$-reduction in pseudo-valuation domains.  Section 3 solves an open problem raised in \cite{HKM} on whether the finite $t$-basic and $v$-basic ideal properties are distinct. We prove that these two notions coincide in any arbitrary domain.  Section 4 features the main result, which establishes the transfer of the finite $t$-basic (equiv., $v$-basic) ideal property to pullbacks in line with Fontana-Gabelli's result on P$v$MDs \cite[Theorem 4.1]{FG} and Gabelli-Houston's result on $v$-domains \cite[Theorem 4.15]{GH}. This allows us to enrich the literature with new families of examples, which put  the class of domains subject to the finite $t$-basic ideal property strictly between the two classes of $v$-domains and integrally closed domains.

For a full treatment of the topic of reduction theory, we refer the reader to \cite{HS2}. For more details about star operations, we refer the reader to \cite{FHP1} and \cite[Sections 32 and 34]{G}.

\section{$t$-Reductions in pseudo-valuation domains}\label{r}

\noindent We first recall the definitions of $t$-reduction and related concepts such as the trivial $t$-reduction and (finite) $t$-basic ideal property.

\begin{definition}[\cite{HKM,KK}]
Let $R$ be a domain and $I$ a nonzero ideal of $R$.
\begin{xenumerate}
\item An ideal $J\subseteq I$ is a \emph{$t$-reduction} of $I$ if $(JI^{n})_{t}=(I^{n+1})_{t}$ for some integer $n\geq0$. The ideal $J$ is a \emph{trivial $t$-reduction} of $I$ if $J_{t}=I_{t}$.
\item $I$ is \emph{$t$-basic} if it has no $t$-reduction other than the trivial $t$-reductions.

\item $R$ has the (finite) $t$-basic ideal property if every nonzero (finitely generated) ideal of $R$ is $t$-basic.
\end{xenumerate}
\end{definition}

For any star operation $\star$,  the $\star$-reduction and related concepts are defined likewise. This is not to be confused with Epstein's $c$-reduction  \cite{Ep1,Ep2,Ep3}, which generalizes the original notion of reduction in a different way and was studied in different settings. Namely, let $c$ be a closure operation.  An ideal $J\subseteq I$ is a \emph{$c$-reduction of} $I$ if $J^{c}=I^{c}$. Thus, for $c:=\star$, Epstein's $c$-reduction coincides with the trivial $\star$-reduction.

In the sequel, we will be using the following obvious facts, for nonzero ideals $J\subseteq I$, without explicit mention:
$$J\ \text{is a $t$-reduction of}\ I \Leftrightarrow J\ \text{is a $t$-reduction of}\ I_{t} \Leftrightarrow  J_{t}\ \text{is a $t$-reduction of}\ I_{t}.$$

Recall that $R$ is a pseudo-valuation domain if $R$ is local and shares its maximal ideal with a valuation overring $V$ or, equivalently, if $R$ is a pullback issued from the following diagram
\[\begin{array}{ccl}
R=\varphi^{-1}(k)       & \longrightarrow                       & k\\
\downarrow              &                                       & \downarrow\\
V                       & \stackrel{\varphi}\longrightarrow     & K:=V/M
\end{array}\]
where $(V,M)$ is a valuation domain (with residue field $K$) and $k$ is a subfield of $K$. For the sake of simplicity, we will say that $R$ is a pseudo-valuation domain issued from $(V,M,k)$. For more details on pseudo-valuation domains, see \cite{HH,HH2} and also \cite{BH,BHLP,CD,D,Park}. 

Note that a reduction is necessarily a $t$-reduction; and the converse is not true in general. The next result investigates necessary and sufficient conditions for the notions of reduction and $t$-reduction to coincide in  pseudo-valuation domains. This result can be used readily to provide examples discriminating between the two notions of reduction and $t$-reduction.

\begin{thm}\label{r:pvd1}
Let $R$ be a pseudo-valuation domain issued from $(V,M,k)$ and set $K:=V/M$. Then, the following statements are equivalent:
\begin{enumerate}[label=\rm(\roman*)]
\item For every nonzero ideals $J\subseteq I$, $J$ is a $t$-reduction of $I$ $\Longleftrightarrow$ $J$ is a reduction of $I$.
\item For each $k$-vector subspace $W$ of $K$ containing $k$, $W^{n}$ is a field for some positive integer $n$.
\end{enumerate}
\end{thm}

\begin{proof}
$(i)\Rightarrow(ii)$ Let $W$ be a $k$-vector subspace of $K$ with $k\subsetneqq W\subsetneqq K$. Let $0\neq a\in M$ and consider the ideals of $R$
$$J:=aR\subseteq I:=a\varphi^{-1}(W).$$
Let $r\geq 1$. Then, the fact that $k\subsetneqq W$ yields
$$(R:I^{r}) = a^{-r}\varphi^{-1}(k:W^{r})=a^{-r}M$$
and then $$(I^{r})_{v}=a^{r}M^{-1}=a^{r}V.$$
By \cite[Proposition 4.3]{HZ}, the $t$- and $v$- operations coincide in $R$. Hence, we have
$$(JI)_{t}=(aI)_{t}=aI_{t}=aI_{v}=a^{2}V=(I^{2})_{v}=(I^{2})_{t}$$
and so $J$ is a $t$-reduction of $I$. By (i), $J$ must be a reduction of $I$ and so
$$a^{n+1}\varphi^{-1}(W^{n})=JI^{n}=I^{n+1}=a^{n+1}\varphi^{-1}(W^{n+1})$$
for some positive integer $n$. It follows that $\varphi^{-1}(W^{n})=\varphi^{-1}(W^{n+1})$; i.e., $W^{n}=W^{n+1}$. Therefore $W^{n}=(W^{n})^{2}$ and thus $W^{n}$ is a ring. In particular, let $0\not=\lambda\in K$ and  let $W_{o}:=k+\lambda k$. Then, there is some positive integer $m$ such that
$$\begin{array}{rcl}
k+\lambda k+\dots+\lambda^{m}k      &=      &W_{o}^{m}\\
                                    &=      &W_{o}^{m+1}\\
                                    &=      &k+\lambda k+\dots+\lambda^{m+1}k.
\end{array}$$
So, $\lambda^{m+1}\in k+\lambda k+\dots+\lambda^{m}k$. Therefore $\lambda$ is algebraic over $k$ and thus $K$ is algebraic over $k$. Consequently, $W^{n}$ is a field, as desired.

$(ii)\Rightarrow(i)$ Let $J\subseteq I$ be a $t$-reduction of $I$; i.e., $(JI^{n})_{t}=(I^{n+1})_{t}$ for some positive integer $n$. If $I$ is an ideal of $V$, then both $JI^{n}$ and $I^{n+1}$ are ideals of $V$ so that $JI^{n}$ and $I^{n+1}$ are divisorial ideals of $R$ by \cite[Theorem 2.13]{HH}. Therefore, we obtain
$$JI^{n} = (JI^{n})_{v} =(JI^{n})_{t} = (I^{n+1})_{t} = (I^{n+1})_{v} = I^{n+1}.$$ That is, $J$ is a reduction of $I$. Next, assume that $I$ is not an ideal of $V$. Then, by \cite[Theorem 2.1(n)]{BG}, $I=a\varphi^{-1}(W)$ for some nonzero $a\in M$ and some $k$-vector space $W$ with $k\subseteq W\subset K$. Assume that $k = W$; i.e., $I=aR$. Then $J_{t}=aR$. Now, if $J\subsetneqq aR$, then $a^{-1}J\subsetneqq R$, hence $a^{-1}J\subseteq M$, whence $J\subseteq aM$. Since $M$ is a  divisorial ideal of $R$ \cite[Corollary 5]{HKLM}, we obtain
$$aR=J_{t}\subseteq (aM)_{t}=aM_{t}=aM$$
which is a contradiction. So, necessarily, $J=I$. Next, assume $k\subsetneqq W$. Suppose $J$ is an ideal of $V$. Then $JI^{n}$ would be an ideal of $V$ and hence a divisorial ideal of $R$, yielding
$$a^{n}J=JI^{n}=(JI^{n})_{v}=(JI^{n})_{t}=(I^{n+1})_{t}=(I^{n+1})_{v}=a^{n+1}V,$$
where the last equality is already handled in (i) $\Rightarrow$ (ii). It follows that
$$J=aV=IV\supseteq I\supseteq J.$$
That is, $J=I$ is an ideal of $V$, absurd. Hence, $J$ is not an ideal of $V$. So, since $J\subseteq I$, we may assume that $J=a\varphi^{-1}(F)$, where $F$ is a $k$-vector subspace of $W$. Now by hypothesis, $W^{s}=W^{s+1}$ is a field for some $s\geq 1$. It follows that
$$FW^{s}=W^{s+1}$$
yielding
$$JI^{s}=a^{s+1}\varphi^{-1}(FW^{s})=a^{s+1}\varphi^{-1}(W^{s+1})=I^{s+1}.$$
Hence $J$ is a reduction of $I$, completing the proof of the theorem.
\end{proof}

Note that the condition (ii) in the above result forces $K$ to be algebraic over $k$. In this vein, this fact can be used readily to provide examples of domains where the two notions of reduction and $t$-reduction are distinct.

\begin{example}
Let $R$ be a pseudo-valuation domain issued from $(V,M,k)$ and set $K:=V/M$.
\begin{xenumerate}
\item Assume that $K$ is a transcendental extension of $k$. Then, the notions of reduction and $t$-reduction are distinct in $R$. For instance, pick a transcendental element $\lambda\in K$ over $k$ and let $W:=k+k\lambda$, $I:=a\phi^{-1}(W)$ and $J=:aR$. Then, $J$ is a proper $t$-reduction of $I$, but $I$ is basic in $R$, as seen the proof of (i) $\Rightarrow$ (ii) of the above theorem.
\item Assume that $[K:k]$ is finite. Then for every $k$-submodule $W$ of $K$ with $k\subseteq W\subseteq K$, some power of $W$ is a field, and hence the notions of reduction and $t$-reduction coincide in $R$.
\end{xenumerate}
\end{example}

Given nonzero ideals $J\subseteq I$, if $J_{t}$ is a reduction of $I_{t}$, then $J$ is a $t$-reduction of $I$. The converse is not true in general as shown by
\cite[Example 2.2]{KK} which  consists of a domain containing two $t$-ideals $J\subsetneqq I$ such that $J$ is a $t$-reduction but not a reduction of $I$. The next result provides a class of (integrally closed) pullbacks where the two assumptions are always equivalent.

\begin{proposition}\label{r:pvd2}
Let $R$ be a pseudo-valuation domain and let $J\subseteq I$ be nonzero ideals of $R$. Then, $J$ is a $t$-reduction of $I$ if and only if $J_{t}$ is a reduction of $I_{t}$.
\end{proposition}

\begin{proof}
Sufficiency is trivial. For the necessity, assume $R$ is issued from $(V,M,k)$ and, without loss of generality, $R\subsetneqq V$. Next, let $J$ be a $t$-reduction of $I$. Then, $J_{t}$ is a $t$-reduction of $I_{t}$ and hence we may assume that $J$ and $I$ are both $t$-ideals. So $(JI^{n})_{t} = (I^{n+1})_{t}$, for some integer $n\geq1$. If $I$ is an ideal of $V$, as in the proof of Theorem~\ref{r:pvd1} ((ii)$\Rightarrow$(i)), we get $JI^{n}=(JI^{n})_{t}=(I^{n+1})_{t}=I^{n+1}$;
that is, $J$ is a reduction of $I$. Next, suppose that $I$ is not an ideal of $V$. By \cite[Theorem 2.1(n)]{BG}, $I=a\varphi^{-1}(W)$ for some nonzero $a\in M$ and some $k$-vector space $W$ with $k\subseteq W\subset K:=V/M$. We claim that $k= W$. Otherwise, we would get, via \cite[Proposition 4.3]{HZ}, that
$I=I_{t}=I_{v}= aV$, where the last equality is already handled in the proof of Theorem~\ref{r:pvd1} ((i)$\Rightarrow$(ii)). It follows that $I$ is an ideal of $V$, the desired contradiction. So, necessarily, $k=W$ and then $I=aR$. By \cite[Lemma 1.2]{HKM}, $I$ is $t$-basic; that is, $J=I$, completing the proof.
\end{proof}

The class of Pr\"ufer domains is, so far, the only known class of domains where these two notions of reduction and $t$-reduction coincide. We close this section with the next result, which features necessary conditions for such a coincidence. For this purpose, recall that a domain where the trivial and $w$-operations are the same is said to be a DW-domain \cite{GP,HZ2,Mi}. Common examples of DW-domains are pseudo-valuation domains, Pr\"ufer domains, and quasi-Pr\"ufer domains (i.e., domains with Pr\"ufer integral closure) \cite[Page 190]{FP}.

\begin{proposition}\label{r:coincide}
Let $R$ be a domain where the notions of reduction and $t$-reduction coincide for all ideals of $R$. Then:
\begin{enumerate}
\item Every nonzero prime ideal of $R$ is a $t$-ideal. In particular, $R$ is a DW-domain.
\item $R$ is integrally closed if and only if $R$ has the finite $t$-basic ideal property.
\item $R$ is a P$v$MD  if and only if $R$ is a Pr\"ufer domain.
\end{enumerate}
\end{proposition}

\begin{proof}
(1) Let $P$ be a nonzero prime ideal of $R$. Clearly, $P$ is a $t$-reduction of $P_{t}$. By hypothesis, $P$ is then a reduction of $P_{t}$. But every prime ideal is a $C$-ideal (i.e., it is not a proper reduction of any larger ideal) \cite[Page 58]{H1}. It follows that $P = P_{t}$, as desired. In particular, every maximal ideal of $R$ is a $t$-ideal and, hence, $R$ is a DW-domain by \cite[Proposition 2.2]{Mi}.

(2) Assume that $R$ is integrally closed and let $I$ be a finitely generated ideal of $R$ and $J$ a $t$-reduction of $I$. By hypothesis, $J$ is a reduction of $I$. So, by a combination of \cite[Corollary 1.2.5]{HS2} and \cite[Proposition 2.2(iii)]{Mi2}, we get $I \subseteq \overline{J} \subseteq J_{t}$, where $\overline{J}$ denotes the integral closure of $J$. It follows that $J_{t} = I_{t}$; i.e., $I$ is $t$-basic, as desired. The converse is true for any arbitrary domain $R$ by \cite[Lemma 1.3]{HKM}.

(3) Assume $R$ is a P$v$MD. By hypothesis, the notions of reduction and $t$-reduction coincide in $R$ and, hence, $R$ is a DW-domain by (1) above. By \cite[Theorem 1.2]{GP}, $R$ is a Pr\"ufer domain. The converse is trivial.
\end{proof}

\section{Equivalence of the finite $t$- and $v$-basic ideal properties}

\noindent For the reader's convenience, recall that a domain is called a $v$-domain if all its nonzero finitely generated ideals are $v$-invertible; an excellent reference for $v$-domains is Fontana \& Zafrullah's comprehensive survey paper \cite{FZ}. Also, recall from \cite{HKM} the following diagram of implications, which puts into perspective the finite basic ideal property for each of the $t$-, $v$-, and $w$-operations:
\bigskip

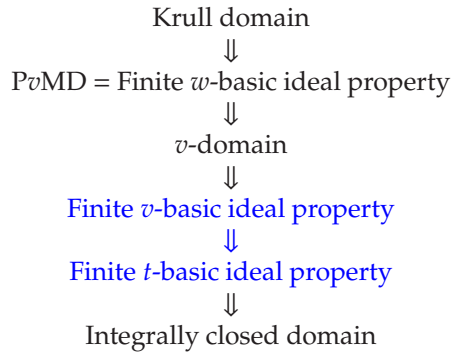
\begin{center}
Krull domain\\
$\Downarrow$\\
P$v$MD = Finite $w$-basic ideal property\\
$\Downarrow$\\
$v$-domain\\
$\Downarrow$\\
{\color{blue} Finite $v$-basic ideal property\\
$\Downarrow$\\
Finite $t$-basic ideal property}\\
$\Downarrow$\\
Integrally closed domain
\end{center}
\bigskip

The problem of whether the fourth implication is reversible was left open in \cite[Section 3]{HKM}. The main result of this section (Theorem~\ref{b:main1}) solves this open problem. For this purpose, recall from \cite{KK} the following: Let $R$ be a domain and $I$ a nonzero ideal of $R$. An element $x\in R$ is $t$-integral over $I$ if there is an equation $x^{n}+a_{1}x^{n-1}+...+a_{n-1}x+a_{n}=0\ \mbox{ with }\ a_{i}\in (I^{i})_{t}\ \forall i=1,...,n$. Consider the two sets:
$$\widetilde{I}:=\big\{x\in R \mid x\ \text{is $t$-integral over}\ I\big\}$$
$$\widehat{I}:=\big\{x\in R \mid I\ \text{is a $t$-reduction of}\ (I,x)\big\}.$$
$\widetilde{I}$ is called the $t$-integral closure of $I$ and is an integrally closed ideal \cite[Theorem 3.2]{KK}, on the other hand, it is not
known if, in general, $\widehat{I}$ is an ideal (see Question 3.5 below). We always have
$$I_{t}\subseteq\widetilde{I}\subseteq\widehat{I}$$
where the first containment is trivial and the second is asserted by \cite[Proposition 3.7]{KK} and can be strict as shown by \cite[Example 3.10(a)]{KK}. However, for the trivial operation,
 it is well-known that the equality $\widetilde{I}=\widehat{I}$ always holds \cite[Corollary 1.2.2]{HS2}; this fact was used to show that the integral closure of an ideal is an ideal \cite[Corollary 1.3.1]{HS2}. Finally, in order to put Theorem~\ref{b:main1} into perspective, recall the following important result.

\begin{thm}[{\cite[Theorem 3.5]{KK}}]
For a domain $R$, the following two assertions are equivalent:
\begin{enumerate}[label=\rm(\roman*)]
\item $I_{t}=\widetilde{I}$ for each nonzero (finitely generated) ideal $I$ of $R$;
\item $R$ is integrally closed.
\end{enumerate}
\end{thm}

Now, to the main result of this section.

\begin{thm}\label{b:main1}
For a domain $R$, the following assertions are equivalent:
\begin{enumerate}[label=\rm(\roman*)]
\item $I_{t}=\widehat{I}$ for each nonzero (finitely generated) ideal $I$ of $R$;
\item $R$ has the finite $t$-basic ideal property;
\item $R$ has the finite $v$-basic ideal property.
\end{enumerate}
\end{thm}

The proof of this result requires the following two lemmas.

\begin{lemma}[{\cite[Lemma 1.7]{HKM}}]\label{b:fg}
Let $R$ be a domain and let $I$ be a finitely generated ideal of $R$. If $J\subseteq I$ is a $t$-reduction of $I$, then there exists a finitely generated ideal $K\subseteq J$ such that $K$ is a $t$-reduction of $I$.
\end{lemma}

 Note that, for any given $\star$-operation, $\star$-reductions of (integral) ideals can be naturally extended to fractional ideals. The following lemma collects basic results on $\star$-reductions of (fractional) ideals.

\begin{lemma}\label{b:frac}
For a domain $R$, let $K\subseteq J\subseteq I$ and $J'\subseteq I'$ be nonzero fractional ideals of $R$.
\begin{enumerate}
\item  If $J$ and $J'$ are $\star$-reductions of $I$ and $I'$, respectively, then $J+J'$ is a $\star$-reduction of $I+I'$ and $JJ'$ is a $\star$-reduction of $II'$.
\item If $K$ is a $\star$-reduction of $J$ and $J$ is a $\star$-reduction of $I$, then $K$ is a $\star$-reduction of $I$.
\item If $K$ is a $\star$-reduction of $I$, then $J$ is a $\star$-reduction of $I$.
\item $J$ is a $\star$-reduction of $I$ $\Leftrightarrow$ $J^{n}$  is a $\star$-reduction of $I^{n}$.
\item  If $J=(a_{1}, ...,a_{k})$, then: $J$ is a $\star$-reduction of $I$ $\Leftrightarrow$ $(a_{1}^{n}, ...,a_{k}^{n})$  is a $\star$-reduction of $I^{n}$.
\end{enumerate}
\end{lemma}

\begin{proof}
Substitute ``$\star$" for ``$t$" and ``fractional ideals" for ``(integral) ideals" in the proofs of \cite[Lemmas 2.5, 2.6 and 2.7]{KK}.
\end{proof}

\begin{proof}[Proof of Theorem~\ref{b:main1}]
In view of the aforementioned diagram, we only need to prove (i) $\Leftrightarrow$ (ii) $\Rightarrow$ (iii). First, let us prove that if the equality  $\widehat{I}=I_{t}$ holds for all nonzero finitely generated ideals then it holds for all nonzero ideals. Indeed, let $I$ be an ideal of $R$ and $x \in R$ such that $I$ is a $t$-reduction of $(I,x)$. So,
$$(I(I,x)^{n})_{t} = ((I,x)^{n+1})_{t}$$
for some positive integer $n$. Hence, $x^{n+1} \in (I(I,x)^{n})_{t}$. Whence, $x^{n+1} \in A_{v}$ for some finitely generated ideal $A\subseteq I(I,x)^{n}$. Moreover, there exist finitely generated subideals $F_{o},F_{1}\dots,F_{n}$ of $I$ such that
$$A \subseteq F_{o}(F_{1},x)(F_{2},x)\cdots(F_{n},x).$$
 Set $F:=\sum_{i=o}^{n} F_{i} \subseteq I$. Then, $A \subseteq F(F,x)^{n}$ and so
 $$x^{n+1} \in (F(F,x)^n)_{v}=(F(F,x)^n)_{t}.$$
 It follows that
 $$((F,x)^{n+1})_{t}=(F(F,x)^{n},x^{n+1})_{t}\subseteq (F(F,x)^{n})_{t}.$$
 Thus,  $F$ is a $t$-reduction of $(F,x)$. Since $F$ is finitely generated, then by hypothesis $x\in\widehat{F}=F_{t}\subseteq I_{t}$. Consequently, $\widehat{I}\subseteq I_{t}$ and, as mentioned above, the reverse inclusion always holds by \cite[Proposition 3.7]{KK}.

Assume that $R$ has the finite $t$-basic ideal property and let $I$ be a finitely generated ideal of $R$ and $x \in \widehat{I}$. Necessarily, $I_{t} = (I,x)_{t}$ which forces $x\in I_{t}$.  Consequently, $\widehat{I}= I_{t}$. Conversely, assume that (i) holds. Let $I:=(a_{1},\dots,a_{n})$ be a nonzero finitely generated ideal of $R$ ($n\geq1$) and  let $J$ be a $t$-reduction of $I$. By Lemma~\ref{b:fg}, we may assume that $J$ is finitely generated. Clearly, we have
$$J \subseteq (J,a_{1},\dots,a_{n-1}) \subseteq I.$$
By \cite[Lemma 2.6]{KK}, $(J,a_{1},\dots,a_{n-1})$ is a $t$-reduction of $I$ which can be regarded as $\big((J,a_{1},\dots,a_{n-1}),a_{n}\big)$. Hence, by hypothesis, $a_{n}\in \widehat{(J,a_{1},\dots,a_{n-1})}=(J,a_{1},\dots,a_{n-1})_{t}$. It follows that
$$I_{t}=(J,a_{1},\dots,a_{n-1})_{t}.$$
But $J$, being a $t$-reduction of $I_{t}$, is also a $t$-reduction of $(J,a_{1},\dots,a_{n-1})$. Therefore, we re-iterate the above process by removing one generator at each step. Eventually, we get $I_{t}=J_{t}$, as desired. This proves (i) $\Leftrightarrow$ (ii).

Assume that $R$ has the finite $t$-basic ideal property and let $I$ be a finitely generated ideal of $R$ and $J$ a $v$-reduction of $I$. So
$$J_{v} = \bigcap_{\lambda \in \Lambda}(a_{\lambda})$$
where the $(a_{\lambda})$'s are the nonzero principal fractional ideals of $R$ containing $J$ by \cite[Theorem 34.1]{G}. By Lemma~\ref{b:frac}, $(a_{\lambda})=(J,a_{\lambda})$ is a $v$-reduction of $(I,a_{\lambda})$ for each $\lambda \in \Lambda$. Hence $(a_{\lambda})$ is a $t$-reduction of $(I,a_{\lambda})$ as both ideals are finitely generated.  Since $R$ has the finite $t$-basic ideal property, one can easily verify that every nonzero \emph{fractional} ideal of $R$ is $t$-basic. Hence, $(a_{\lambda}) = (I,a_{\lambda})_{t}$ for each $\lambda \in \Lambda$. Therefore
$$I_{v} = I_{t} \subseteq \bigcap_{\lambda \in \Lambda}(a_{\lambda}) = J_{v}.$$
Hence, $I_{v}=J_{v}$; that is, $I$ is $v$-basic.  This proves (ii) $\Rightarrow$ (iii), completing the proof of the theorem.
\end{proof}

New examples of domains subject to the finite $t$-basic (equiv., $v$-basic) ideal property will be provided in the next section.  We close this section with the following open question:

\begin{question}
Let $I$ be a nonzero ideal, is $\widehat{I}$ always an ideal?
\end{question}

\section{Transfer of the finite $t$-basic ideal property to pullbacks}

\noindent Let us fix notation for this section. Let $T$ be a domain, $M$ a maximal ideal of $T$, $K$ its residue field, $\varphi:T\longrightarrow K$ the canonical surjection, and $D$ a proper subring of $K$ with quotient field $k$. Let $R$ be the pullback issued from the following diagram of canonical homomorphisms:
\[\begin{array}{cccl}
                    &R                  & \longrightarrow                       & D\\
(\ \square\ )       &\downarrow         &                                       & \downarrow\\
                    &T                  & \stackrel{\varphi}\longrightarrow     & K=T/M.
\end{array}\]

So, $R:=\varphi^{-1}(D)\subsetneqq T$.  This section establishes necessary and sufficient conditions for a pullback  of type $\square$  issued from local domains to inherit the finite $t$-basic (equiv., $v$-basic) ideal property. Recall, at this point, that a domain with the $t$-basic ideal property is completely integrally closed \cite[Proposition 1.4]{HKM}.  Therefore, by \cite[Lemma 26.5]{G}, a pullback of type $\square$ never has the $t$-basic ideal property.

It is worthwhile recalling that the finite $t$-basic ideal property lies between the two notions of $v$-domain and integrally closed domain \cite{HKM}; and that the finite $w$-basic ideal property coincides with the  P$v$MD property \cite[Theorem 2.1]{HKM}. Also, the transfer of the notions of P$v$MD and $v$-domain to pullbacks was established, respectively, by Fontana \& Gabelli in \cite{FG} and by Gabelli \& Houston in \cite{GH}, which summarizes as follows:

\begin{thm}[{\cite[Theorem 4.1]{FG} \& \cite[Theorem 4.15]{GH}}]
Let $R$ be a pullback of type $\square$. Then, $R$ is P$v$MD (resp., $v$-domain) if and only if $T$ and $D$ are P$v$MDs (resp., $v$-domains), $T_{M}$ is a valuation domain, and $k=K$.
\end{thm}

Finally, recall that if $T$ is integrally closed, then the integral closure of $R$ is $\varphi^{-1}(\overline{D})$, where $\overline{D}$ denotes the integral
closure of $D$ in $K$. This follows easily from the fact that $R$ and $T$ have the same quotient field. Next, we announce the main result of this section which allows us to enrich the literature with new families of examples, putting  the new class of domains subject to the finite $t$-basic ideal property strictly  between the two classes of $v$-domains and integrally closed domains.

\begin{thm}\label{b:main2}
Let $R$ be a pullback of type $\square$ such that $T$ is local. Then, $R$ has the finite $t$-basic ideal property if and only if $T$ and $D$ have the finite $t$-basic ideal property and $k=K$.
\end{thm}

\begin{proof}
Assume that $R$ has the finite $t$-basic ideal property. We first prove that $k= K$. Assume, by way of contradiction, that $k \subsetneqq K$. By \cite[Proposition 2.4]{GH}, there is an element $x \in T \setminus R$ with $(R:(1,x)) = M$. Hence
$$x^{2} (R:(1,x)) = x^{2}M \subseteq TM \subseteq R;\ \text{ i.e.,}\ x^{2} \in (1,x)_{v}.$$
Therefore, for any nonzero $m\in M$, we have
$$x^{2}m^{2} \in (m^{2},xm^{2})_{v}=(m^{2},xm^{2})_{t}$$
and so
$$((m,xm)^{2})_{t}=(m^{2}, xm^{2})_{t}=(m(m,xm))_{t}$$
forcing $(m)$  to be a $t$-reduction of $(m,xm)$ in $R$. Whence, $(m,xm)_{t} = (m)$. It follows that $xm \in (m)$ and thus $x \in R$, the desired contradiction.
Next, we prove that $T$ has the finite $t$-basic ideal property. Below, we denote by $v_{1}$ and $t_{1}$ the $v$- and $t$- operations with respect to $T$. Let $I$ be a nonzero finitely generated proper ideal of $T$ and $J$ a $t$-reduction of $I$. So $(JI^n)_{t_{1}} = (I^{n+1})_{t_{1}}$ for some positive integer $n$. We may assume, by Lemma~\ref{b:fg}, that $J$ is finitely generated. If $(I^{n+1})_{v_{1}}$ is principal; say, $(I^{n+1})_{t_{1}}=(I^{n+1})_{v_{1}} = (a)$ for some nonzero $a \in T$, then
$$aJ_{t_{1}} = (JI^{n+1})_{t_{1}} = (I^{n+2})_{t_{1}} = aI_{t_{1}}$$
yielding $J_{t_{1}} = I_{t_{1}}$. Next, suppose that $(JI^n)_{v_{1}} = (I^{n+1})_{v_{1}}$ is not principal. Since $k=K$, then $T$ is a localization of $R$ (cf. \cite{F,Kab}). So, $J=BT$ and $I=AT$, for some nonzero finitely generated ideals $B\subseteq A$ of $R$. By \cite[Proposition 2.7(1)(b)]{GH}, we obtain
$$(A^{n+1})_{t}=(A^{n+1})_{v}=(I^{n+1})_{v_{1}}=(I^{n+1})_{t_{1}}=(JI^n)_{t_{1}}=(JI^n)_{v_{1}}=(BA^n)_{v}=(BA^n)_{t}.$$
It follows that $B$ is a $t$-reduction of $A$ and thus $B_{t}=A_{t}$. By \cite[Lemma 3.4]{Kg2}, we get
$$J_{t_{1}} = (B_{t}T)_{t_{1}} = (A_{t}T)_{t_{1}} = I_{t_{1}}.$$
Therefore, in both cases, we showed that $J$ is a trivial $t$-reduction of $I$, as desired.
Next, we show that $D$ has the finite $t$-basic ideal property. Let $A$ be a nonzero finitely generated ideal of $D$ and let $B$ be a $t$-reduction of $A$. Let $t_{D}$ denote the $t$-operation with respect to $D$. So, $(BA^{n})_{t_{D}} = (A^{n+1})_{t_{D}}$ for some positive integer $n$. We may assume, by Lemma~\ref{b:fg}, that $B$ is finitely generated. By \cite[Corollary 1.7]{FG}, $I:=\varphi^{-1}(A)$ and $J:=\varphi^{-1}(B)$ are two nonzero finitely generated ideals of $R$ (containing $M$). Since $k=K$, by \cite[Proposition 1.6(a) \& Proposition 1.8(a3)]{FG}, we obtain
$$(JI^{n})_{t}=(\varphi^{-1}(BA^{n}))_{t}=\varphi^{-1}((BA^{n})_{t_{D}})=\varphi^{-1}( (A^{n+1})_{t_{D}})=(\varphi^{-1}( A^{n+1}))_{t}=(I^{n+1})_{t}.$$
 Hence $J$ is a $t$-reduction of $I$ and thus $J_{t}=I_{t}$. It follows that
$$B_{t_{D}}=\varphi(\varphi^{-1}(B_{t_{D}}))=\varphi(J_{t})=\varphi(I_{t})=\varphi(\varphi^{-1}(A_{t_{D}}))=A_{t_{D}}$$
completing the proof of the ``only if" assertion.

Conversely, assume that $T$ and $D$ have the finite $t$-basic ideal property and $k=K$. Notice that, in presence of the latter assumption, $M$ cannot be finitely generated \cite[Lemma 4.1]{GH}. Also, recall that we always have $M_{v}=M$ \cite[Corollary 5]{HKLM}. Next,  let $I$ be a nonzero finitely generated ideal of $R$ and let $J$ be a finitely generated subideal of $I$ with  $(JI^{n})_{t} = (I^{n+1})_{t}$ for some positive integer $n$. By \cite[Proposition 1.6]{GH2}, any ideal of $R$ is comparable to $M$. So, we envisage two cases:\\
{\bf Case 1:} Suppose that $M\subsetneqq I$.  We first claim that $M \subsetneqq I^{n+1}$; otherwise, $I^{n+1} \subseteq M$ yields, by \cite[Proposition 1.1]{FG}, $T = (IT)^{n+1} = I^{n+1}T \subseteq MT = M$, absurd. Moreover, we have $M \subsetneqq J$; otherwise, we would have
$$J \subseteq M \subsetneqq I^{n+1}\subseteq J_{t}=J_{v},$$
which is absurd. Further, we claim that $M \subsetneqq JI^{n}$; otherwise, $JI^{n} \subseteq M$ yields via \cite[Proposition 1.1]{FG}
$$T=(JT)(IT)^{n}=(JI^{n})T\subseteq MT=M,$$
which is absurd. Now, let $A:=\varphi(I)$ and $B:=\varphi(J)$, two nonzero finitely generated ideals of $D$. Therefore, by \cite[Proposition 1.6(b) \& Proposition 1.8(b3)]{FG}, we get
$$(BA^{n})_{t_{D}}=(\varphi(JI^{n}))_{t_{D}}=\varphi((JI^{n})_{t})=\varphi((I^{n+1})_{t})=(\varphi(I^{n+1}))_{t_{D}}=(A^{n+1})_{t_{D}}.$$
Hence $B$ is a $t$-reduction of $A$ and thus $B_{t_{D}}=A_{t_{D}}$. It follows that
$$J_{t}=\varphi^{-1}(\varphi(J_{t}))=\varphi^{-1}(B_{t_{D}})=\varphi^{-1}(A_{t_{D}})=\varphi^{-1}(\varphi(I_{t}))=I_{t}.$$
{\bf Case 2:} Suppose that $I\subsetneqq M$. If $II^{-1} \nsubseteq M$, then there is a nonzero $x\in\qf(R)$ with $M \subsetneqq xI\subseteq R$, hence $xJ_{t} = xI_{t}$ by Case 1, whence $J_{t} = I_{t}$. So, we may assume $II^{-1} \subseteq M$. Now, note that $(JI^n)^{-1} = (I^{n+1})^{-1}$. So, by \cite[Proposition 2.2(1)]{GH2}, we have
$$\begin{array}{rcl}
(JI^nT)_{t_{1}} &=      &(JI^nT)_{v_{1}}\\
                &=      &((JI^nT)^{-1})^{-1}\\
                &=      &((JI^n)^{-1}T)^{-1}\\
                &=      &((I^{n+1})^{-1}T)^{-1}\\
                &=      &((I^{n+1}T)^{-1})^{-1}\\
                &=      &(I^{n+1}T)_{v_{1}}\\
                &=      &(I^{n+1}T)_{t_{1}}.
\end{array}$$
Hence $JT$ is a $t$-reduction of $IT$. It follows, via \cite[Proposition 2.2(1)]{GH2}, that
$$J^{-1}T=(JT)^{-1}=((JT)_{v_{1}})^{-1}=((JT)_{t_{1}})^{-1}=((IT)_{t_{1}})^{-1}=((IT)_{v_{1}})^{-1}=(IT)^{-1}=I^{-1}T.$$
On the other hand, the assumption $II^{-1} \subseteq M$ yields
$$(IT)(IT)^{-1}=II^{-1}T\subseteq MT=M.$$
Hence $IT$ is not invertible and, a fortiori, not principal in $T$. Therefore, by \cite[Proposition 2.7(a)]{GH}, we get
$$J^{-1}\subseteq J^{-1}T = I^{-1}T = (IT)^{-1} = (M:I) = I^{-1}\subseteq J^{-1}.$$
Consequently, $I_{t} = I_{v} = J_{v} = J_{t}$, completing the proof of the theorem.
\end{proof}

Theorem~\ref{b:main2} allows us to enrich the literature with new families of examples, which put the class of domains subject to the finite $t$-basic ideal property strictly between the two classes of integrally closed domains and $v$-domains.

\begin{example}\label{b:ex1}
Consider any non-trivial pseudo-valuation domain $R$ issued from $(V,M,k)$ with $k$ algebraically closed in $K:=V/M$. Then, $R$ is an integrally closed domain by \cite[Theorem 2.1]{BG}, which does not have the finite $t$-basic ideal property by Theorem~\ref{b:main2}. Moreover, the two notions of reduction and $t$-reduction are distinct in $R$ by Proposition~\ref{r:coincide}(2).
\end{example}

\begin{example}\label{b:ex2}
Consider any pullback $R$ of type $\square$ issued from $(T,M,D)$ where $\qf(D)=T/M$, $T$ is a non-valuation local $v$-domain, and $D$ is a $v$-domain.  Then, $R$ has the finite $t$-basic ideal property by \cite[Proposition 1.6]{HKM} and Theorem~\ref{b:main1} and Theorem~\ref{b:main2}, which is not a $v$-domain by \cite[Theorem 4.15]{GH}. One can easily build non-valuation local $v$-domains via pullbacks through \cite[Theorem 4.15]{GH}.
\end{example}

Here is a specific example, where we ensure, moreover, that the two notions of reduction and $t$-reduction are distinct.

\begin{example}\label{b:ex2-1}
Let $T:=\Q(X)[[Y, Z]]=\Q(X)+M$ and $R:=\Z[X]+M$. Clearly, $T$ and $D:=\Z[X]$ have the finite $t$-basic property (since they are Noetherian Krull domains). By   Theorem~\ref{b:main2}, $R$ has the finite $t$-basic property. Also $R$ is not a $v$-domain since $T$ is a non-valuation local domain. Next, let $0\not=a\in\Z$ and consider the finitely generated ideal of $R$ given by $I:=(a, X)\Z[X]+M=aR+XR$. Clearly $I^{-1}=R$ and so $(I^{s})^{-1}=R$, for every positive integer $s$. In particular, we have
$$(I^{2}I)_{t}=(I^{3})_{t}=(I^{3})_{v}=R=(I^{2})_{v}=(I^{2})_{t}$$
and hence $I^{2}$ is a $t$-reduction of $I$. However, $I^{2}$ is not a reduction of $I$; otherwise, if $I^{n+2}=I^{2}I^{n}=I^{n+1}$, for some $n\geq 1$, this would contradict \cite[Theorem 76]{Ka}. It follows that the notions of reduction and $t$-reduction are distinct in $R$, as desired.
\end{example}

We close this section with the following two open questions.

\begin{question}
Is Theorem~\ref{b:main2} valid for the classical pullbacks $R=D+M$ issued from $T:=K+M$ not necessarily local? The idea here is that (since $k=K$, then) $T=S^{-1}R$ for $S:=D\setminus \{0\}$. Clearly, the current proof of the ``only if" assertion works for this context.
\end{question}

\begin{question}
Is Theorem~\ref{b:main2} valid for the non-local case through an additional assumption on $T_{M}$? The idea here is that, ``($k=K$ and hence) $R_{M}=T_{M}$" is a necessity for the finite $t$-basic property; and for the P$v$MD and $v$-domain notions, $R_{M}=T_{M}$ is a valuation domain. So, one needs to investigate this localization for the $t$-basic ideal property in this context.
\end{question}

\end{document}